\documentclass[preprint,3p,11pt]{elsarticle}
\usepackage{url}
\usepackage{amssymb}
\usepackage{amsmath}
\usepackage{stmaryrd}
\usepackage{siunitx}
\usepackage{commath}
\usepackage{xcolor}
\usepackage[caption=false]{subfig}
\usepackage{enumerate}
\usepackage{cleveref}
\journal{Computer Methods in Applied Mechanics and Engineering}
\makeatletter
\newcommand{\tnorm}{\@ifstar\@tnorms\@tnorm}
\newcommand{\@tnorms}[1]{%
  \left|\mkern-1.5mu\left|\mkern-1.5mu\left|
   #1
  \right|\mkern-1.5mu\right|\mkern-1.5mu\right|
}
\newcommand{\@tnorm}[2][]{%
  \mathopen{#1|\mkern-1.5mu#1|\mkern-1.5mu#1|}
  #2
  \mathclose{#1|\mkern-1.5mu#1|\mkern-1.5mu#1|}
}
\makeatother

\newcommand{\jump}[1]{\llbracket #1 \rrbracket}

\newtheorem{theorem}{Theorem}
\newtheorem{lemma}{Lemma}
\newtheorem{corollary}{Corollary}
\newdefinition{remark}{Remark}
\newproof{proof}{Proof}
\begin{document}
\begin{frontmatter}
  \title{An embedded--hybridized discontinuous Galerkin finite element
    method for the Stokes equations}
  \author[SR]{Sander Rhebergen\corref{cor1}\fnref{label1}}
  \ead{srheberg@uwaterloo.ca}
  \fntext[label1]{\url{https://orcid.org/0000-0001-6036-0356}}
  \author[GNW]{Garth N. Wells\fnref{label2}}
  \ead{gnw20@cam.ac.uk}
  \address[SR]{Department of Applied Mathematics, University of
    Waterloo, Canada} \address[GNW]{Department of Engineering,
    University of Cambridge, United Kingdom}
  \fntext[label2]{\url{https://orcid.org/0000-0001-5291-7951}}
\begin{abstract}
  We present and analyze a new embedded--hybridized discontinuous
  Galerkin finite element method for the Stokes problem. The method has
  the attractive properties of full hybridized methods, namely an $H({\rm
  div})$-conforming velocity field, pointwise satisfaction of the
  continuity equation and \emph{a priori} error estimates for the
  velocity that are independent of the pressure. The
  embedded--hybridized formulation has advantages over a full hybridized
  formulation in that it has fewer global degrees-of-freedom for a given
  mesh and the algebraic structure of the resulting linear system is
  better suited to fast iterative solvers. The analysis results are
  supported by a range of numerical examples that demonstrate rates of
  convergence, and which show computational efficiency gains over a full
  hybridized formulation.
\end{abstract}
\end{frontmatter}
\section{Introduction}
\label{sec:introduction}

Hybridized discontinuous Galerkin (HDG) methods were introduced with
the purpose of reducing the computational cost of discontinuous
Galerkin methods while retaining the attractive features. HDG methods
for the Stokes equations were introduced in \citep{Cockburn:2009c} for
the vorticity-velocity-pressure formulation of the Stokes problem, and
a modified version of this method for the velocity-pressure-gradient
formulation of the Stokes equations was introduced and analyzed
in~\citep{Cockburn:2011, Cockburn:2014b, Nguyen:2010}. An HDG method
for the velocity-pressure formulation of the Stokes equations was
analyzed in \citep{Rhebergen:2017}. To lower the computational cost of
HDG methods, embedded discontinuous Galerkin (EDG) \citep{Guzey:2007}
methods for incompressible flows have been developed which retain many
of the attractive features of discontinuous Galerkin methods but with
the same number of global degrees of freedom as a continuous Galerkin
method on a given mesh~\citep{labeur:2007, Labeur:2012}. The main
difference between EDG and HDG methods is the choice of function
spaces for the facet Lagrange multipliers. In the case of an HDG
method, the Lagrange multipliers are discontinuous between facets. To
reduce the number of degrees-of-freedom for a given mesh one may use
continuous Lagrange multipliers, leading to an EDG method.

The HDG method reduces the computational cost of discontinuous
Galerkin methods by introducing facet variables and eliminating local
(cell-wise) degrees-of-freedom, following ideas originally introduced
for mixed finite element methods, e.g.~\citep{Boffi:book}. This static
condensation can significantly reduce the size of the global problem
for higher-order discretizations. It is possible to reduce the problem
size of $H({\rm div})$-conforming HDG methods further by exploiting
that the normal component of the velocity is continuous across
facets. These methods only require to enforce continuity in the
tangential direction of the facet velocity \citep{Lehrenfeld:2010,
  Lehrenfeld:2016}. By the \emph{projected jumps} method, in which the
polynomial degree of the tangential facet velocity is reduced by one
compared to the cell velocity approximation, \citep{Lehrenfeld:2010,
  Lehrenfeld:2016} were able to lower the number of globally coupled
degrees-of-freedom even more. An alternative to increase the
performance of HDG methods is by a cell-wise post-processing. For
diffusion dominated, incompressible flows, post-processing techniques
have been introduced that result in super-convergent,
$H({\rm div})$-conforming, and point-wise divergence-free velocity
fields, see for example \citep{Cesmelioglu:2013, Cockburn:2011}.

A property that we are particularly interested in is \emph{pressure
robustness}, which is when the \emph{a priori} error estimate for the
velocity does not depend on the pressure error (scaled by the inverse of
the viscosity). A way to achieve pressure robustness is to devise a
finite element method with an $H({\rm div})$-conforming and
divergence-free velocity approximation. In a series of papers,
\citet{Cockburn:2002, Cockburn:2003, Cockburn:2004b, Cockburn:2007b}
introduced an $H({\rm div})$-conforming discontinuous Galerkin method
for incompressible flows (see also \citep{Wang:2007}). $H({\rm
div})$-conforming and divergence-free HDG methods were introduced and
analysed for incompressible flows by, for example \citep{Fu:2019,
Lehrenfeld:2010, Lehrenfeld:2016, Lederer:2017, Rhebergen:2017,
Rhebergen:2018a}. For other $H({\rm div})$-conforming finite element
methods we refer to, for example, \citep{Brennecke:2015, Lederer:2017b,
Linke:2016a, Linke:2016b, Schroeder:2018} and the review paper by
\citet{John:2017}. When a method is not $H({\rm div})$-conforming,
reconstruction operators \citep{Lederer:2017b, Linke:2016b} or
post-processing \citep{Cockburn:2004b, Guzman:2016} of the approximate
velocity field can be used to achieve pressure robustness.

Embedded discontinuous Galerkin (EDG) methods use continuous Lagrange
multipliers on facets, thereby reducing the number of globally coupled
degrees-of-freedom compared to HDG methods. However, some attractive
properties are lost. It was shown, for example, that there is no
post-processing of EDG methods that result in super-convergent solutions
\citep{Cockburn:2009}. Moreover, the approximate velocity field is not
$H({\rm div})$-conforming and as consequence the EDG method is not
pressure robust. In the context of the incompressible Navier--Stokes
equations, mass and momentum conservation cannot be satisfied
simultaneously~\citep{Labeur:2012}.

The number of globally coupled degrees-of-freedom alone is not
necessarily a good proxy for efficiency; efficiency will also depend on
the performance of linear solvers. We formulate and analyze a new
EDG--HDG method that retains the pressure-robustness property of HDG
methods, and has the efficiency characteristics of EDG methods. The
method yields a velocity field that is pointwise divergence-free and
automatically $H({\rm div})$-conforming. This is achieved through
hybridization via a facet pressure field that is discontinuous between
facets, as is typical for HDG methods. For the facet velocity field, we
use a continuous basis. This is desirable for substantially reducing the
number of global velocity degrees-of-freedom on a given mesh, and
continuous methods are generally observed to lead to better performance
of preconditioned iterative solvers. We present analysis for the EDG,
EDG--HDG and HDG formulations of the Stokes problem in a unified
setting. In particular, the analysis highlights key differences between
the methods in the context of pressure robustness.

We test performance numerically using the preconditioner developed and
analyzed in~\citep{Rhebergen:2018b}. As anticipated, the preconditioner
is more effective in terms of lower iteration counts for the more
regular EDG--HDG method compared to the HDG method. The fewer linear
solver iterations combined with the fewer global degrees-of-freedom for
the EDG--HDG method compared to the HDG method lead to the observation
via numerical experiments that the EDG--HDG method is considerably more
efficient in terms of the time required to reach a given discretization
error. It should be mentioned that our numerical experiments only
compare the `standard' EDG, HDG and EDG--HDG methods without taking
advantage of any modifications that can be made to reduce the problem
size even further, such as the projected jumps method. As such, these
numerical experiments serve to indicate the speed-up that is possible
when exploiting the `continuous' structure built into the EDG and
EDG--HDG methods compared to the HDG method.

The remainder of this article is organized as follows. In
\cref{sec:stokes} we introduce the HDG, EDG and EDG--HDG methods for the
Stokes problem, and we prove inf-sup stability for all three methods.
Error estimates are provided in \cref{sec:error_estimates}, and in
particular pressure robustness of the HDG and EDG--HDG methods is
considered. The error estimates are supported by numerical examples in
\cref{sec:numerical_examples}. Preconditioning is discussed in
\cref{sec:preconditioning}, with performance of the different methods
with preconditioned solvers examined by numerical examples. Conclusions
are drawn in \cref{sec:conclusions}.

\section{The hybridized, embedded and embedded-hybridized discontinuous
  Galerkin method}
\label{sec:stokes}

In this section we consider the embedded, hybridized, and
embedded--hybridized discontinuous Galerkin methods for the Stokes
problem:
\begin{subequations}
  \label{eq:stokes}
  \begin{align}
    -\nu\nabla^2u + \nabla p &= f & & \mbox{in} \ \Omega,
    \\
    \nabla\cdot u &= 0 & & \mbox{in}\ \Omega,
    \\
    u &= 0 & & \mbox{on}\ \partial\Omega,
    \\
    \int_{\Omega} p \dif x &= 0,
  \end{align}
\end{subequations}
where $\Omega \subset \mathbb{R}^d$ is a polygonal ($d = 2$) or
polyhedral ($d = 3$) domain, $u : \Omega \to \mathbb{R}^d$ is the
velocity, $p:\Omega \to \mathbb{R}$ is the pressure, $f : \Omega \to
\mathbb{R}^d$ is the prescribed body force, and $\nu \in \mathbb{R}^+$
is a given constant kinematic viscosity.

\subsection{Notation}

Let $\mathcal{T} := \{K\}$ be a triangulation of $\Omega$. This
triangulation consists of non-overlapping simplicial cells $K$. The
length measure of a cell $K$ is denoted by $h_K$. The outward unit
normal vector, on the boundary of a cell, $\partial K$, is denoted
by~$n$. An interior facet $F$ is shared by two adjacent cells $K^+$ and
$K^-$ while a boundary facet is a facet of $\partial K$ that lies on
$\partial \Omega$. The set and union of all facets are denoted by,
respectively, $\mathcal{F} = \{F\}$ and~$\Gamma_0$.

We consider the following discontinuous finite element function spaces
on~$\Omega$:
\begin{equation}
  \begin{split}
    V_h  &:= \cbr{v_h\in \sbr[1]{L^2(\Omega)}^d
        : \ v_h \in \sbr{P_k(K)}^d, \ \forall\ K\in\mathcal{T}},
    \\
    Q_h &:= \cbr{q_h\in L^2(\Omega) : \ q_h \in P_{k-1}(K) ,\
    \forall \ K \in \mathcal{T}},
  \end{split}
  \label{eqn:spaces_cell}
\end{equation}
where $P_k(K)$ denotes the set of polynomials of degree $k$ on a
cell~$K$. On $\Gamma_{0}$ we consider the finite element spaces
\begin{equation}
  \begin{split}
    \bar{V}_h &:= \cbr{\bar{v}_h \in \sbr[1]{L^2(\Gamma_0)}^d : \ \bar{v}_h \in
      \sbr{P_{k}(F)}^d \ \forall \ F \in \mathcal{F}, \ \bar{v}_h
      = 0 \ \mbox{on}\ \partial\Omega},
    \\
    \bar{Q}_h &:= \cbr{\bar{q}_h \in L^2(\Gamma_0) : \ \bar{q}_h \in
      P_{k}(F) \ \forall\ F \in \mathcal{F}},
  \end{split}
  \label{eqn:HDG_spaces_facet}
\end{equation}
where $P_k(F)$ denotes the set of polynomials of degree $k$ on a
facet~$F$. We also introduce the extended function spaces
\begin{align}
  V(h) &:= V_h + \sbr[1]{H_0^1(\Omega)}^d \cap \sbr[1]{H^2(\Omega)}^d,
  \\
  Q(h) &:= Q_h + L_0^2(\Omega) \cap H^1(\Omega),
\end{align}
and
\begin{align}
  \bar{V}(h) &:= \bar{V}_h + [H_0^{3/2}(\Gamma_0)]^d,
  \\
  \bar{Q}(h) &:= \bar{Q}_h + H^{1/2}_0(\Gamma_0).
\end{align}

We define two norms on $V(h) \times \bar{V}(h)$, namely,
\begin{equation}
\label{eq:stability_norm}
  \tnorm{ \boldsymbol{v} }_v^2 := \sum_{K\in\mathcal{T}}\norm{\nabla v }^2_{K}
  + \sum_{K\in\mathcal{T}} \frac{\alpha_{v}}{h_K}\norm{\bar{v} - v}^2_{\partial K},
\end{equation}
and
\begin{equation}
  \label{eq:boundedness_norm}
  \tnorm{ \boldsymbol{v} }_{v'}^2 := \tnorm{ \boldsymbol{v} }_v^2
  + \sum_{K\in\mathcal{T}}\frac{h_K}{\alpha_v}\norm{\frac{\partial v}{\partial n}}^2_{\partial K},
\end{equation}
where $\alpha_v > 0$ is a penalty parameter that will be defined
later. From the discrete trace
inequality~\cite[Remark~1.47]{Pietro:book},
\begin{equation}
  \label{eq:trace_inequality}
  h_K^{1/2}\norm{v_h}_{\partial K}\le C_t\norm{v_h}_{K}
  \qquad \forall v_h \in P_k(K),
\end{equation}
where $C_{t}$ depends on $k$, spatial dimension and cell shape, it
follows that the norms $\tnorm{\cdot}_v$ and $\tnorm{\cdot}_{v'}$ are
equivalent on the finite element space~$V_h \times \bar{V}_{h}$:
\begin{equation}
  \label{eq:equivalentNorms_v_vprime}
  \tnorm{ \boldsymbol{v}_h }_{v} \le \tnorm{ \boldsymbol{v}_h }_{v'}
  \le
  c(1+\alpha_v^{-1})\tnorm{ \boldsymbol{v}_h }_{v}
  \quad \forall \boldsymbol{v}_h \in V_h \times \bar{V}_{h},
\end{equation}
where $c > 0$ a constant independent of $h$,
see~\cite[Eq.~(5.5)]{Wells:2011}.

On $\bar{Q}(h)$ and $Q(h) \times \bar{Q}(h)$ we introduce, respectively,
\begin{equation}
  \label{eq:def_normbarq}
    \norm{\bar{q}}_{p}^2 := \sum_{K\in\mathcal{T}}h_K\norm{\bar{q}}^2_{\partial K},
    \qquad
    \tnorm{\boldsymbol{q}}^2_p := \norm{q}^2_{\Omega} + \norm{\bar{q}}_{p}^2.
\end{equation}

\subsection{Weak formulation}
\label{sec:weak_forms}

Consider the bilinear form
\begin{equation}
  \label{eq:Bh}
  B_h((\boldsymbol{u}_h, \boldsymbol{p}_h), (\boldsymbol{v}_h, \boldsymbol{q}_h))
  :=
  a_h(\boldsymbol{u}_h, \boldsymbol{v}_h) + b_h(\boldsymbol{p}_h, v_h)
      - b_h(\boldsymbol{q}_h, u_h),
\end{equation}
where
\begin{subequations}
  \begin{align}
    \label{eq:formA}
    a_h(\boldsymbol{u}, \boldsymbol{v}) :=&
    \sum_{K\in\mathcal{T}}\int_{K}\nu \nabla u : \nabla v \dif x
    + \sum_{K\in\mathcal{T}}\int_{\partial K}\frac{\nu \alpha_{v}}{h_K}(u - \bar{u})\cdot(v - \bar{v}) \dif s
    \\
    \nonumber
    &- \sum_{K\in\mathcal{T}}\int_{\partial K}\nu \sbr{(u-\bar{u})\cdot \frac{\partial v}{\partial n}
    + \frac{\partial u}{\partial n}\cdot(v-\bar{v})} \dif s,
    \\
    \label{eq:formB}
    b_h(\boldsymbol{p}, v) :=&
    - \sum_{K\in\mathcal{T}}\int_{K}p \nabla \cdot v \dif x
    + \sum_{K\in\mathcal{T}}\int_{\partial K}v\cdot n\bar{p} \dif s.
  \end{align}
  \label{eq:bilin_forms}
\end{subequations}
The methods involve: find $(\boldsymbol{u}_h, \boldsymbol{p}_h) \in X_{h}$
such that
\begin{equation}
  \label{eq:discrete_problem}
  B_h((\boldsymbol{u}_h, \boldsymbol{p}_h), (\boldsymbol{v}_h, \boldsymbol{q}_h))
  = \int_{\Omega} f \cdot v_h \dif x
  \qquad \forall (\boldsymbol{v}_h, \boldsymbol{q}_h) \in X_{h},
\end{equation}
where $X_{h} = X_{h}^{v} \times X_{h}^{q}$, and the
different formulations use the following spaces:
\begin{equation}
  X_{h}
  :=
  \begin{cases}
    \del{V_{h} \times \bar{V}_{h}} \times \del{Q_{h} \times \bar{Q}_{h}} & \text{HDG method},
    \\
    \del{V_{h} \times (\bar{V}_{h} \cap C^{0}(\Gamma_{0})} \times \del{Q_{h} \times \bar{Q}_{h}}
        & \text{EDG--HDG method},
    \\
    \del{V_{h} \times (\bar{V}_{h} \cap C^{0}(\Gamma_{0})}
      \times \del{Q_{h} \times (\bar{Q}_{h} \cap C^{0}(\Gamma_{0})} & \text{EDG method}.
  \end{cases}
\label{eq:method_spaces}
\end{equation}
The HDG method uses facet function spaces that are discontinuous. In the
EDG--HDG method the facet velocity field is continuous and the facet
pressure field is discontinuous, and in the EDG method both velocity and
pressure facet functions are continuous.

All three formulations yield computed velocity fields that are
pointwise solenoidal on cells. This is an immediate consequence of
$\nabla \cdot v_h \in Q_h$ for all $v_h \in V_h$. For the HDG and
EDG--HDG formulations the facet pressure is discontinuous (lying in
$\bar{Q}_{h}$), in which case it is straightforward to show that
$u_{h} \in H({\rm div}, \Omega)$, i.e.~the normal component of the
computed velocity $u_{h}$ is continuous across cell facets
\citep{Rhebergen:2017}. We will therefore refer to the HDG and
EDG--HDG methods as being \emph{$H({\rm div})$-conforming}. In the
case of the EDG method, the normal component of the velocity is only
weakly continuous across cell facets.

The HDG variant of the formulation was analyzed in
\citep{Rhebergen:2017}, but pressure robustness, in which the \emph{a
  priori} error estimate for the velocity does not depend on the
pressure error, was not proven. We generalize and extend analysis
results to include the EDG and EDG--HDG formulations, and to prove
pressure robustness of the HDG and EDG--HDG formulations. In the
following analysis we present, where possible, results that hold
without reference to a specific method. Where this is not possible we
comment explicitly on the conditions for a result to hold for a
specific method.

\subsection{Consistency}

We consider the following space for the exact solution to the Stokes
problem:
\begin{equation}
  X := \del{\sbr[1]{H_0^1(\Omega)}^d \cap \sbr[1]{H^2(\Omega)}^d} \times
  \del{L_0^2(\Omega) \cap H^1(\Omega)}.
\end{equation}
\begin{lemma}
  \label{lem:consistency}
  Let $(u, p) \in X$ solve the Stokes problem \cref{eq:stokes} and let
  $\boldsymbol{u} = (u, u)$ and $\boldsymbol{p} = (p, p)$. Then
  \begin{equation}
    \label{eq:consistency}
    B_h((\boldsymbol{u}, \boldsymbol{p}), (\boldsymbol{v}_h, \boldsymbol{q}_h))
    =
    \int_{\Omega} f \cdot v_h \dif x \quad
    \forall (\boldsymbol{v}_h,\boldsymbol{q}_h) \in X_h.
  \end{equation}
\end{lemma}
\begin{proof}
  See \cite[Lemma 3.1]{Rhebergen:2017}. \qed
\end{proof}
The result in \cite[Lemma~3.1]{Rhebergen:2017} holds trivially for the
EDG--HDG and EDG methods as they involve subspaces of the HDG method.

\subsection{Stability and boundedness}
\label{ss:stab_bound}

Stability and boundedness of the vector-Laplacian term, $a_h$, for the
EDG and EDG--HDG method is a direct consequence of the stability and
boundedness results of the HDG method proven in~\citep{Rhebergen:2017}.
We state these results here for completeness.
\begin{lemma}[Stability of $a_h$]
  \label{lem:stab_ah}
  There exists a $\beta_{v} > 0$, independent of $h$, and a constant
  $\alpha_{0} > 0$ such that for $\alpha_{v} > \alpha_{0}$ and for
  all~$\boldsymbol{v}_{h} \in V_{h} \times \bar{V}_{h}$
  \begin{equation}
    \label{eq:stab_ah}
    a_h(\boldsymbol{v}_h, \boldsymbol{v}_h) \ge \nu \beta_v \tnorm{\boldsymbol{v}_h}^2_v.
  \end{equation}
\end{lemma}
\begin{proof}
  See \cite[Lemma~4.2]{Rhebergen:2017}. \qed
\end{proof}
\begin{lemma}[Boundedness of $a_h$]
  \label{lem:bound_ah}
  There exists a $c > 0$, independent of $h$, such that for all
  $\boldsymbol{u}\in V(h) \times \bar{V}(h)$ and for all
  $\boldsymbol{v}_h \in V_{h} \times \bar{V}_{h}$
  \begin{equation}
    \label{eq:bound_ah}
    \envert{a_h(\boldsymbol{u}, \boldsymbol{v}_h)}
    \le C_a \nu \tnorm{\boldsymbol{u}}_{v'} \tnorm{\boldsymbol{v}_h}_v,
  \end{equation}
  with $C_a = c(1+\alpha_v^{-1/2})$.
\end{lemma}
\begin{proof}
  See~\cite[Lemma~4.3]{Rhebergen:2017}. \qed
\end{proof}
\Cref{lem:stab_ah,lem:bound_ah} hold trivially for the EDG--HDG and
EDG formulations as velocity fields in both cases are subspaces
of~$V_{h} \times \bar{V}_{h}$.

The velocity-pressure coupling in \cref{eq:discrete_problem} is
provided by:
\begin{equation}
  \label{eq:formB_b1b2}
  b_h(\boldsymbol{p}_h, v_h) := b_1(p_h, v_h) + b_2(\bar{p}_h, v_h),
\end{equation}
where
\begin{equation}
  \label{eq:bh1bh2}
  b_1(p_h, v_h) := -\sum_{K\in\mathcal{T}} \int_K p_h \nabla \cdot v_h \dif x
  \quad \text{and} \quad
  b_2(\bar{p}_h, v_h) := \sum_{K\in\mathcal{T}}\int_{\partial K} v_h \cdot n \bar{p}_h \dif s.
\end{equation}
We consider the inf-sup condition for $b_{1}$ and $b_{2}$ separately
first, after which we prove inf-sup stability for~$b_{h}$.

It is useful to introduce the Brezzi--Douglas--Marini (BDM) finite
element space, $V_h^{\rm BDM}$ (see \citep{Boffi:book}):
\begin{equation}
  \begin{split}
    V_h^{\rm BDM}(K)
    &:= \cbr{ v_h \in \sbr{P_k(K)}^d : v_h\cdot n \in L^2(\partial K),\
      v_h\cdot n|_F\in P_k(F)},
    \\
    V_h^{\rm BDM}
    &:= \cbr{v_h \in H({\rm div};\Omega) : v_h|_K \in V_h^{\rm BDM}(K) \
      \forall K \in \mathcal{T}},
  \end{split}
\end{equation}
and the following interpolation operator~\cite[Lemma~7]{Hansbo:2002}.
\begin{lemma}
  \label{lem:BDMprojection}
  If the mesh consists of triangles in two dimensions or tetrahedra in
  three dimensions there is an interpolation operator $\Pi_{\rm BDM} :
  [H^1(\Omega)]^d \to V_h$ with the following properties for all $u\in
  [H^{k+1}(K)]^d$ where $k \ge 1$:
  \begin{enumerate}[(i)]
    \item $\jump{n\cdot\Pi_{\rm BDM} u} = 0$, where $\jump{a} =
      a^{+} +a^{-}$ and $\jump{a}=a$ on, respectively, interior and
      boundary faces is the usual jump operator.
    \item
      \label{item:bdm_inverse_inequality}
    $\norm{u-\Pi_{\rm BDM} u}_{m,K} \le ch_K^{l - m} \norm{u}_{l,K}$ with $m = 0, 1, 2$
    and $m\le l \le k+1$.
    \item $\norm{\nabla\cdot(u - \Pi_{\rm BDM} u)}_{m, K} \le c
      h_K^{l - m} \norm{\nabla\cdot u}_{l, K}$ with $m = 0, 1$ and $m
      \le l \le k$.
    \item \label{item:bdm_div}
    $\int_K q(\nabla\cdot u - \nabla\cdot\Pi_{\rm BDM}u)
      \dif x = 0$ for all $q \in P_{k - 1}(K)$.
    \item $\int_F \bar{q}(n\cdot u - n\cdot \Pi_{\rm BDM} u) \dif s
      = 0$ for all $\bar{q}\in P_k(F)$, where $F$ is a face on~$\partial
      K$.
  \end{enumerate}
\end{lemma}
We will also use an interpolation operator $\mathcal{I}_{h} :
H^1(\Omega) \to V_{h} \cap C^{0}(\bar{\Omega})$ with the following
property:
\begin{equation}
  \label{eq:interpolationError}
  \sum_{K} h^{-2}_{K} \norm{v - \mathcal{I}_{h}v}_{0, K}^{2}
    \le c \envert{v}^{2}_{1, \Omega},
\end{equation}
for example, the Scott--Zhang interpolant (see
\cite[Theroem~4.8.12]{brenner:book}).

\begin{lemma}[Stability of $b_1$]
  \label{thm:stab_b1}
  There exists a constant $\beta_{1} > 0$, independent of $h$, such
  that for all~$q_h \in Q_h$
  \begin{equation}
    \label{eq:stab_b1}
    \beta_{1} \norm{q_h}_{\Omega}
    \le
    \sup_{\boldsymbol{v}_h \in V_h^{\rm BDM} \times \del{\bar{V}_{h} \cap C^{0}(\Gamma_{0})}}
    \frac{ b_1(q_h, v_h) }{\tnorm{ \boldsymbol{v}_h }_{v}}.
  \end{equation}
\end{lemma}
\begin{proof}
  We first consider a bound for $\tnorm{ \cdot }_{v}$. From
  \cref{item:bdm_inverse_inequality} of \cref{lem:BDMprojection} and the
  triangle inequality,
  \begin{equation}
    \label{eq:elementbdmbound}
    \norm{\nabla(\Pi_{\rm BDM}v)}_{K}
    \le
    \norm{\nabla v - \nabla(\Pi_{\rm BDM} v)}_{K} +
    \norm{\nabla v}_{K}
    \le
    c \norm{v}_{1, K},
  \end{equation}
  and
  \begin{equation}
    \label{eq:trace_bound}
    \sum_{K} h^{-1} \norm{\Pi_{\rm BDM} v - \mathcal{I}_{v} v }_{0, \partial K}^{2}
    \le
   \sum_{K} c h^{-2} \norm{\Pi_{\rm BDM} v - \mathcal{I}_{v} v }_{0, K}^{2}
   \le
    c \norm{v}_{1, K}^{2}
  \end{equation}
  for all $v \in \sbr{H^{1}(\Omega)}^{d}$, where in
  \cref{eq:trace_bound} the first inequality is due to the trace
  inequality \cref{eq:trace_inequality}, and the second is due to
  \cref{item:bdm_inverse_inequality}
  of \cref{lem:BDMprojection} and the interpolation estimate in
  \cref{eq:interpolationError}. Combining
  \cref{eq:elementbdmbound,eq:trace_bound},
  \begin{equation}
    \label{eq:tnormbdmfinalbound}
    \tnorm{(\Pi_{\rm BDM}v, \mathcal{I}_h v)}_{v}^{2}
      \le c (1 + \alpha_v) \norm{v}^2_{1, \Omega}
    \quad \forall v \in \sbr[1]{H^{1}(\Omega)}^{d}.
  \end{equation}

  For all $q \in L_{0}^{2}(\Omega)$ there exists a $v_{q} \in
  \sbr{H^{1}_{0}(\Omega)}^{d}$ such that
  \begin{equation}
    \label{eq:stab_b}
    q = \nabla \cdot v_q
    \quad \text{and} \quad
    \beta_c \norm{v_{q}}_{1, \Omega} \le \norm{q}_{\Omega},
  \end{equation}
  where $\beta_{c} > 0$ is a constant depending only on $\Omega$ (see,
  e.g.~\cite[Theorem~6.5]{Pietro:book}). For $q_{h} \in Q_{h}$, we
  denote $v_{q_h} \in \sbr{H_0^1(\Omega)}^d$ such that $\nabla \cdot
  v_{q_h} = q_h$. It then follows that
  \begin{equation}
    \label{eq:normqhinBDMbh}
    \norm{q_h}^2_{\Omega}
    = \int_{\Omega}q_h\nabla\cdot v_{q_h} \dif x
    = \int_{\Omega}q_h\nabla\cdot \Pi_{\rm BDM} v_{q_{h}} \dif x
    = -b_1(q_h, \Pi_{\rm BDM}v_{q_h})
  \end{equation}
  by \cref{item:bdm_div} of \cref{lem:BDMprojection} and by the
  definition of $b_{1}$ in \cref{eq:bh1bh2},
  and from \cref{eq:tnormbdmfinalbound},
  \begin{equation}
    \label{eq:tnormbdmfinalbound_q}
    \tnorm{(\Pi_{\rm BDM} v_{q_{h}}, \mathcal{I}_h v_{q_{h}})}_{v}
    \le
    c \sqrt{1 + \alpha_v} \norm{v_{q_{h}}}_{1, \Omega}
    \le
    c \sqrt{1 + \alpha_v} \beta_c^{-1} \norm{q_{h}}_{\Omega}.
  \end{equation}

  Satisfaction of \cref{eq:stab_b1} follows from
  \begin{equation}
    \label{eq:fininfsupql2}
    \sup_{ \boldsymbol{v}_h\in V_h^{\rm BDM} \times \del{\bar{V}_h \cap C^{0}(\Gamma_{0)}}}
    \frac{ -b_1(q_h, v_h) }{\tnorm{ \boldsymbol{v}_h }_v}
    \ge
    \frac{ -b_1(q_h, \Pi_{\rm BDM} v_{q_h}) }{\tnorm{ (\Pi_{\rm BDM}v_{q_h}, \mathcal{I}_h^kv_{q_h}) }_v}
    \ge
    \frac{\beta_c}{c\sqrt{1+\alpha_v}}\norm{q_h}_{\Omega},
  \end{equation}
  where $\nabla \cdot v_{q_h} = q_{h}$, and where
  \cref{eq:normqhinBDMbh,eq:tnormbdmfinalbound_q} are used for the
  second inequality. \qed
\end{proof}
The preceding proof is simpler and more general than
\cite[Lemma~4.4]{Rhebergen:2017}, which was for the case of
discontinuous facet functions, i.e.,~$\boldsymbol{v}_h \in V_h^{\rm BDM}
\times \bar{V}_{h}$.

\begin{lemma}[Stability of $b_2$]
  \label{thm:stab_b2}
  There exists a constant $\beta_{2} > 0$, independent of $h$, such that
  for all $\bar{q}_h \in \bar{Q}_h$
  \begin{equation}
    \label{eq:stab_b2}
    \beta_{2} \norm{\bar{q}_h}_{p}
    \le
    \sup_{\boldsymbol{v}_h \in V_h \times \del{\bar{V}_{h} \cap C^{0}(\Gamma_{0})}}
      \frac{ b_{2}(\bar{q}_h, v_h) }{\tnorm{ \boldsymbol{v}_h }_{v}}.
  \end{equation}
\end{lemma}
\begin{proof}
  Note that
  \begin{equation}
    \label{eq:stab_b2_d}
    \beta_{2} \norm{\bar{q}_h}_{p}
    \le
    \sup_{v_h\in V_h}
    \frac{ b_2(\bar{q}_h, v_h) }{\tnorm{ (v_h, 0) }_{v}}
    \le
    \sup_{\boldsymbol{v}_h \in V_h \times \del{\bar{V}_{h} \cap C^{0}(\Gamma_{0})}}
    \frac{ b_2(\bar{q}_h, v_h) }{\tnorm{ \boldsymbol{v}_h }_{v}},
  \end{equation}
  where the first inequality was proven
  in~\cite[Lemma~3]{Rhebergen:2018b}. \qed
\end{proof}
\begin{lemma}[Boundedness of $b_{1}$ and $b_{2}$]
  \label{lem:boundedness_b12}
  There exists a $C_b > 0$, independent of~$h$, such that for all
  $\boldsymbol{v}_h \in V_h \times \bar{V}_h$ and for all
  $\boldsymbol{q}_h \in Q_h \times \bar{Q}_h$
  \begin{equation}
    \label{eq:boundedness_b12}
    \envert{b_{1}(q_h, v_h)} \le
    \tnorm{ \boldsymbol{v}_h }_v \tnorm{ \boldsymbol{q}_h }_p
    \quad \text{and} \quad
    \envert{b_{2}(\bar{q}_h, v_h)} \le
    C_b \tnorm{ \boldsymbol{v}_h }_v \tnorm{ \boldsymbol{q}_h }_p.
  \end{equation}
\end{lemma}
\begin{proof}
  The proof is identical to that of \cite[Lemma 4.8]{Rhebergen:2017}. \qed
\end{proof}

\Cref{lem:boundedness_b12} holds trivially for the EDG--HDG and EDG
formulations as the velocity and pressure fields in both cases are
subspaces of $V_h \times \bar{V}_h$ and $Q_h \times \bar{Q}_h$,
respectively.

The following is a reduced version of~\cite[Theorem~3.1]{Howell:2011}
and  will be used to prove stability of the combined pressure coupling
term.
\begin{theorem}
  \label{thm:howellWalkington}
  Let $U$, $P_1$, and $P_2$ be reflexive Banach spaces, and let $b_1:
  P_1 \times U \to \mathbb{R}$, and $b_2 : P_2 \times U \to \mathbb{R}$
  be bilinear and bounded. Let
  \begin{equation}
    Z_{b_2} := \cbr{ v \in U : b_2(p_i, v) = 0 \quad \forall p_2 \in P_2} \subset U,
    \label{eq:Zb_space}
  \end{equation}
  then the following are equivalent:
  \begin{enumerate}
  \item There exists $c > 0$ such that
    \begin{equation*}
      \label{eq:combined_inf_sup}
      \sup_{v\in U} \frac{b_1(p_1, v) + b_2(p_2, v)}{\norm{v}_U} \ge c\del{\norm{p_1}_{P_1} + \norm{p_2}_{P_2}}
      \quad (p_1, p_2) \in P_1 \times P_2.
    \end{equation*}
  \item There exists $c > 0$ such that
    \begin{equation*}
      \label{eq:proofOverZb2U}
      \sup_{v \in Z_{b_2}} \frac{b_1(p_1, v)}{\norm{v}_U} \ge c \norm{p_1}_{P_1}, \ p_1 \in P_1
      \ \mbox{and} \
      \sup_{v \in U} \frac{b_2(p_2, v)}{\norm{v}_U} \ge c \norm{p_2}_{P_2}, \ p_2\in P_2.
    \end{equation*}
  \end{enumerate}
\end{theorem}

\begin{lemma}[Stability of $b_h$]
  \label{lem:stab_bh}
  There exists constant $\beta_p > 0$, independent of $h$, such that for
  all~$\boldsymbol{q}_h \in Q_h \times \bar{Q}_{h}$
  \begin{equation}
    \label{eq:stab_bh}
    \beta_p \tnorm{\boldsymbol{q}_h}_{p} \le \sup_{\boldsymbol{v}_h \in V_{h}
      \times (\bar{V}_{h} \cap C^{0}(\Gamma_{0}))}
    \frac{ b_h(\boldsymbol{q}_h, v_h) }{\tnorm{ \boldsymbol{v}_h }_{v}}.
  \end{equation}
\end{lemma}
\begin{proof}
  Let $b_{1}(\cdot, \cdot)$ and $b_2(\cdot, \cdot)$ be defined as in
  \cref{eq:bh1bh2}, and let $U := V_{h} \times (\bar{V}_{h} \cap
  C^{0}(\Gamma_{0}))$, $P_1 := Q_h$, $P_{2} := \bar{Q}_{h}$ and $Z_{b_2}
  := V_h^{{\rm BDM}} \times (\bar{V}_h \cap C^{0}(\Gamma_{0}))$. This
  definition of $Z_{b_2}$ satisfies \cref{eq:Zb_space} by virtue of
  continuity of the normal component of functions in $V_h^{{\rm BDM}}$
  across facets. The conditions in \cref{eq:proofOverZb2U} of
  \cref{thm:howellWalkington} are satisfied by
  \cref{thm:stab_b1,thm:stab_b2}. The result follows by equivalence of
  \cref{eq:proofOverZb2U,eq:combined_inf_sup} in
  \cref{thm:howellWalkington}. \qed
\end{proof}
\Cref{lem:stab_bh} is posed for the EDG--HDG case, but holds trivially
for the HDG case with a larger facet velocity space, $\boldsymbol{v}_{h}
\in V_{h} \times \bar{V}_{h} \supset V_h \times (\bar{V}_h \cap
C^{0}(\Gamma_{0}))$, and for the EDG method with a smaller facet
pressure space, $\boldsymbol{q}_{h} \in Q_{h} \times (\bar{Q}_{h} \cap
C^{0}(\Gamma_{0})) \subset Q_{h} \times \bar{Q}_{h}$.

An immediate consequence of the stability of $a_h$ (\cref{lem:stab_ah})
and the stability of $b_h$ (\cref{lem:stab_bh}) is that the discrete
problem in \cref{eq:discrete_problem} is well-posed, see,
e.g.~\cite[Theorem~2.4]{Gatica:book}.

\section{Error estimates and pressure robustness}
\label{sec:error_estimates}

Convergence is an immediate consequence of the stability and
boundedness results. Let
$(u, p) \in \sbr{H^{k + 1}(\Omega)}^{d} \times H^{k}(\Omega)$
$k \ge 1$ solve the Stokes problem \cref{eq:stokes}, and let
$\boldsymbol{u} = (u, u)$ and $\boldsymbol{p} = (p, p)$. If
$(\boldsymbol{u}_h, \boldsymbol{p}_h) \in X_h$ solves the finite
element problem in \cref{eq:discrete_problem}, then there exists a
constant $c > 0$, independent of $h$, such that
\begin{equation}
  \label{eq:error_tnorm}
  \nu^{1/2} \tnorm{\boldsymbol{u} - \boldsymbol{u}_h}_v
  + \nu^{-1/2} \tnorm{\boldsymbol{p} - \boldsymbol{p}_h}_p
  \le
  c\del{h^k \nu^{1/2} \norm{u}_{k + 1, \Omega} + h^{k} \nu^{-1/2} \norm{p}_{k, \Omega}}.
\end{equation}
A proof of this estimate is a simple extension of the proof given for
the HDG discretization of the Stokes problem
in~\cite[Section~5]{Rhebergen:2017}.

The error estimate in \cref{eq:error_tnorm} involves norms of the
velocity and pressure fields, and concerningly the norm of the exact
pressure scaled by~$\nu^{-1/2}$. For the HDG and the EDG--HDG cases, but
not for the EDG case, an improved estimate for the velocity field can be
found that does not depend on the pressure. The improved estimate relies
on the velocity field being pointwise divergence-free and $H({\rm
div})$-conforming (the latter condition not being met by the EDG
method).
\begin{theorem}[Pressure robust error estimate]
  \label{thm:pressure_robust}
  Let $u \in \sbr{H^{k + 1}(\Omega)}^d$ be the velocity solution of
  the Stokes problem \cref{eq:stokes} with $k \ge 1$, let
  $\boldsymbol{u} = (u, u)$, and let $\boldsymbol{u}_h \in X_h^{v}$ be
  the velocity solution of the finite element problem
  \cref{eq:discrete_problem} for the HDG or EDG--HDG
  formulations. There exists a constant $C > 0$, independent of $h$,
  such that
  \begin{equation}
    \label{eq:pressure_robust}
    \tnorm{\boldsymbol{u} - \boldsymbol{u}_h}_v \le C h^k \norm{u}_{k+1, \Omega}.
  \end{equation}
\end{theorem}
\begin{proof}
  Consider
  $\boldsymbol{w}_h := \boldsymbol{u}_h - \boldsymbol{v}_h \in
  X_{h}^{v}$, subject to
  $b_h(\boldsymbol{q}_h, w_h) = 0 \, \forall \boldsymbol{q}_h \in
  X_{h}^q$. From \cref{lem:stab_ah,lem:bound_ah} it holds that for all
  $\boldsymbol{v}_h$:
  \begin{equation}
    \label{eq:tnormwh}
    \begin{split}
      \beta_v \nu \tnorm{\boldsymbol{w}_h}_v^2
      &\le a_h(\boldsymbol{w}_h, \boldsymbol{w}_h)
      \\
      &= a_h(\boldsymbol{u} - \boldsymbol{v}_h, \boldsymbol{w}_h) + a_h(\boldsymbol{u}_h - \boldsymbol{u}, \boldsymbol{w}_h)
      \\
      &\le C_a \nu \tnorm{\boldsymbol{u} - \boldsymbol{v}_h}_{v'} \tnorm{\boldsymbol{w}_h}_v + a_h(\boldsymbol{u}_h, \boldsymbol{w}_h) - a_h(\boldsymbol{u}, \boldsymbol{w}_h)
      \\
      &= C_a \nu \tnorm{\boldsymbol{u} - \boldsymbol{v}_h}_{v'} \tnorm{\boldsymbol{w}_h}_v
      + \int_{\Omega} f \cdot w_h \dif x - b_h(\boldsymbol{p}_h, w_h) - a_h(\boldsymbol{u}, \boldsymbol{w}_h).
    \end{split}
  \end{equation}
  The condition $b_h(\boldsymbol{q}_h, w_h) = 0 \ \forall \
  \boldsymbol{q}_h \in X_{h}^q$ on $w_{h}$ implies that
  $b_h(\boldsymbol{p}, w_h) = b_h(\boldsymbol{p}_h, w_h) = 0$, and with
  consistency (\cref{lem:consistency}) it also follows that
  $\int_{\Omega} f \cdot w_h \dif x - a_h(\boldsymbol{u},
  \boldsymbol{w}_h) = 0$. Therefore, from \cref{eq:tnormwh} we have
  $\tnorm{\boldsymbol{w}_h}_v \le (C_a/\beta_v) \tnorm{\boldsymbol{u} -
  \boldsymbol{v}_h}_{v'}$ and
  \begin{equation}
    \label{eq:tnormwh_i}
    \tnorm{\boldsymbol{u} - \boldsymbol{u}_h}_v
    \le
    \tnorm{\boldsymbol{u} - \boldsymbol{v}_h}_v + \tnorm{\boldsymbol{w}_h}_v
    \le
    \del{1 + \frac{C_a}{\beta_v}} \tnorm{\boldsymbol{u} - \boldsymbol{v}_h}_{v'}.
  \end{equation}
  This leads to
  \begin{equation}
    \label{eq:tnormwh_f}
    \tnorm{\boldsymbol{u} - \boldsymbol{u}_h}_v
    \le c \inf_{\substack{\boldsymbol{v}_h \in X_h^{v} \\ b_{h}(\boldsymbol{q}_{h}, v_{h}) = 0
       \, \forall \boldsymbol{q} \in X_{h}^{q} }} \tnorm{\boldsymbol{u} - \boldsymbol{v}_h}_{v'}
   \le C h^k\norm{u}_{k+1,\Omega},
  \end{equation}
  where the second inequality follows from setting $\boldsymbol{v}_h =
  \Pi \boldsymbol{u} = (\Pi_{\rm BDM} u, \Pi_{L^{2}(\Gamma_{0})} u)$,
  where $\Pi_{L^{2}(\Gamma_{0})}$ is the $L^{2}$-projection into the
  facet velocity space, and the application of the BDM interpolation in
  \cref{lem:BDMprojection} and standard polynomial interpolation and
  trace inequality estimates (see \cref{sec:interpolation_estimate} for
  the interpolation estimate). \qed
\end{proof}

The refined estimate shows that (i) the velocity error does not depend
on the pressure, and as a consequence, (ii) the velocity error does not
depend on the viscosity. Formulations in which the velocity error
estimate is independent of the pressure are sometimes called
\emph{pressure robust}~\citep{John:2017,Lederer:2017}.

\begin{remark}[Lack of pressure robustness for the EDG case] For the EDG
  case, the analysis supporting \cref{thm:pressure_robust} breaks down
  due to the jump in the normal component of $w_{h}$ not being zero
  across cell facets. If $w_{h}$ is chosen to satisfy
  $b_h(\boldsymbol{q}_h, w_h) = 0 \, \forall \boldsymbol{q}_h \in X_h^q
  = Q_h \times \bar{Q}_{h} \cap C^{0}(\Gamma_{0})$, \cref{eq:tnormwh}
  holds but $b_h(\boldsymbol{p}, w_h) \ne 0$ and the step to
  \cref{eq:tnormwh_i} breaks down.  For the EDG case we have:
  \begin{equation}
    \begin{split}
      \beta_v \nu \tnorm{\boldsymbol{w}_h}_v^2
      &\le C_a \nu \tnorm{\boldsymbol{u} - \boldsymbol{v}_h}_{v'} \tnorm{\boldsymbol{w}_h}_v
      + \int_{\Omega} f \cdot w_h \dif x - b_h(\boldsymbol{p}_h, w_h) - a_h(\boldsymbol{u}, \boldsymbol{w}_h)
      \\
      &\le C_a \nu \tnorm{\boldsymbol{u} - \boldsymbol{v}_h}_{v'} \tnorm{\boldsymbol{w}_h}_v
        + \envert{b_{h}(\boldsymbol{p}, w_{h})}.
    \end{split}
  \end{equation}
  Since $b_{h}(\boldsymbol{p}, w_{h}) = b_{h}(\boldsymbol{p} - \boldsymbol{q}_{h}, w_{h})$ as $w_{h}$
  is chosen such that $b_{h}(\boldsymbol{q}_{h}, w_{h}) = 0$,
  by boundedness of $b_{h}$,
  \begin{equation}
    \tnorm{\boldsymbol{w}_h}_v \le
    \frac{C_a}{\beta} \tnorm{\boldsymbol{u} - \boldsymbol{v}_h}_{v'}
      + \frac{1}{\beta_v\nu}\tnorm{\boldsymbol{p} - \boldsymbol{p}_{h}}_p.
  \end{equation}
  which leads to
  \begin{equation}
    \tnorm{\boldsymbol{u} - \boldsymbol{u}_h}_v
    \le c \inf_{\substack{\boldsymbol{v}_h \in X_h^{v} \\b_{h}(\boldsymbol{q}_{h}, v_{h}) = 0
        \, \forall \boldsymbol{q}_{h} \in X_{h}^{q} }} \tnorm{\boldsymbol{u} - \boldsymbol{v}_h}_{v'}
    + \frac{1}{\beta_v\nu} \inf_{\boldsymbol{q}_h \in X_{h}^{q}}\tnorm{\boldsymbol{p}
    - \boldsymbol{q}_h}_p.
  \end{equation}
  This shows for the EDG method that the velocity error has a dependence
  on $1/\nu$ times the pressure error. \qed
\end{remark}

By adjoint consistency of $a_{h}$ and under appropriate regularity
assumptions, for the HDG and EDG--HDG methods the error estimate
\begin{equation}
  \norm{u - u_h} \le c h^{k + 1} \norm{u}_{k + 1, \Omega}
\end{equation}
follows straightforwardly from application of the Aubin--Nitsche trick.
The analysis is included in \cref{eq:l2normerrorestimate} for
completeness.

\section{Numerical tests}

The performance of the three formulations is considered in terms of
computed errors and solution time when solved using specially
constructed preconditioned iterative solvers. Efficiency is also
assessed in terms of the time required to compute solutions to a
specific accuracy. An element is identified by the formulation type
(HDG/EDG--HDG/EDG) and $P^k$--$P^{k-1}$, where the cell and facet
velocity and facet pressure are approximated by polynomials of degree
$k$, and the cell pressure is approximated by polynomials of
degree~$k - 1$. The penalty parameter is taken as $\alpha_v = 6k^2$ in
2D and $\alpha_v = 10k^2$ in 3D for HDG, and as $\alpha_v = 4k^2$ in
2D and $\alpha_v = 6k^2$ in 3D for EDG and EDG--HDG. The
$k$-dependency is typical of interior penalty
methods~\citep{Riviere:book} and we find the constants to be reliable
across a range of problems. All test cases have been implemented in
MFEM~\citep{mfem-library} with solver support from
PETSc~\citep{petsc-user-ref,petsc-web-page}. When applying algebraic
multigrid, we use the BoomerAMG library~\citep{Henson:2002}.

\subsection{Observed convergence rates}
\label{sec:numerical_examples}

We consider the Kovasznay~\citep{Kovasznay:1948} problem on a domain
$\Omega = \del{-0.5, 1} \times \del{-0.5, 1.5}$, for which the
analytical solution is:
\begin{subequations}
  \label{eq:kovasznay}
  \begin{align}
    u_x &= 1 - e^{\lambda x_{1}} \cos(2 \pi x_{2}), \quad
    \\
    u_y &= \frac{\lambda}{2 \pi} e^{\lambda x_{1}} \sin(2\pi x_{2}), \quad
    \\
    p   &= \frac{1}{2} \del{1 - e^{2\lambda x_{1}}} + C,
  \end{align}
\end{subequations}
where $C$ is an arbitrary constant, and where
\begin{equation}
  \lambda
  = \frac{1}{2 \nu} - \del{\frac{1}{4 \nu^{2}} + 4 \pi^2}^{1/2}.
\end{equation}
We choose $C$ such that the mean pressure on $\Omega$ is zero. Dirichlet
boundary conditions for the velocity on~$\partial \Omega$  interpolate
the analytical solution.

Observed rates of convergence for $\nu = 1/40$ are presented in
\cref{tab:kovasznay} for a series of refined meshes. Optimal rates of
convergence are observed for all test cases, including for the
remarkably simple $P^{1}$--$P^{0}$ case.

\begin{table}
  \caption{Computed velocity, pressure and velocity divergence errors in
    the $L^2$-norm, and rates of convergence, for the Kovasznay flow
    problem for the HDG, EDG and EDG--HDG methods for different orders
    of polynomial approximation.} {\small
    \begin{center}
      \begin{tabular}{cc|cc|cc|c}
        \multicolumn{7}{c}{HDG} \\
        \hline
        Degree & Cells & $\norm{u-u_h}$ & Order & $\norm{p-p_h}$ & Order & $\norm{\nabla\cdot u_h}$ \\
        \hline
        $P^1$--$P^0$
               & 672    & 8.2e-3 & 1.9 & 4.2e-2 & 1.0 & 1.2e-14 \\
               & 2,688  & 2.1e-3 & 2.0 & 2.1e-2 & 1.0 & 2.3e-14 \\
               & 10,752 & 5.3e-4 & 2.0 & 1.1e-2 & 1.0 & 4.6e-14 \\
               & 43,088 & 1.3e-4 & 2.0 & 5.4e-3 & 1.0 & 9.1e-14 \\
        \multicolumn{7}{c}{} \\
        $P^2$--$P^1$
               & 672    & 7.1e-4 & 3.0 & 2.0e-3 & 1.9 & 5.6e-14\\
               & 2,688  & 8.7e-5 & 3.0 & 5.2e-4 & 1.9 & 1.6e-13\\
               & 10,752 & 1.1e-5 & 3.0 & 1.3e-4 & 2.0 & 2.0e-13\\
               & 43,088 & 1.3e-6 & 3.0 & 3.4e-5 & 2.0 & 4.3e-13\\
        \multicolumn{7}{c}{} \\
        $P^5$--$P^4$
               & 672    & 4.3e-8  & 6.0 & 1.9e-7 & 5.0 & 9.1e-13\\
               & 2,688  & 6.7e-10 & 6.0 & 6.0e-9 & 5.0 & 1.6e-12\\
               & 10,752 & 1.7e-11 & 5.3 & 1.9e-10 & 5.0 & 3.4e-12\\
        \hline
        \multicolumn{7}{c}{} \\
        \multicolumn{7}{c}{EDG} \\
        \hline
        Degree & Cells & $\norm{u-u_h}$ & Order & $\norm{p-p_h}$ & Order & $\norm{\nabla\cdot u_h}$ \\
        \hline
        $P^1$--$P^0$
               & 672    & 3.3e-2 & 1.8 & 4.3e-2 & 1.0 & 1.4e-14\\
               & 2,688  & 8.4e-3 & 2.0 & 2.1e-2 & 1.0 & 2.9e-14\\
               & 10,752 & 2.1e-3 & 2.0 & 1.1e-2 & 1.0 & 4.8e-14\\
               & 43,088 & 5.2e-4 & 2.0 & 5.2e-3 & 1.0 & 3.0e-13\\
        \multicolumn{7}{c}{} \\
        $P^2$--$P^1$
               & 672    & 9.0e-4 & 3.0 & 2.5e-3 & 1.8 & 5.0e-14\\
               & 2,688  & 1.1e-4 & 3.0 & 7.1e-4 & 1.8 & 9.9e-14\\
               & 10,752 & 1.4e-5 & 3.0 & 1.9e-4 & 1.9 & 1.9e-13\\
               & 43,088 & 1.7e-6 & 3.0 & 4.9e-5 & 2.0 & 3.8e-13\\
        \multicolumn{7}{c}{} \\
        $P^5$--$P^4$
               & 672    & 4.2e-8  & 6.0 & 1.6e-7  & 5.0 & 4.5e-13\\
               & 2,688  & 6.5e-10 & 6.0 & 5.0e-9  & 5.0 & 9.1e-13\\
               & 10,752 & 1.0e-11 & 6.0 & 1.5e-10 & 5.0 & 1.8e-12\\
        \hline
        \multicolumn{7}{c}{} \\
        \multicolumn{7}{c}{EDG--HDG} \\
        \hline
        Degree & Cells & $\norm{u-u_h}$ & Order & $\norm{p-p_h}$ & Order & $\norm{\nabla\cdot u_h}$ \\
        \hline
        $P^1$--$P^0$
               & 672    & 3.4e-2 & 1.9 & 4.4e-2 & 1.0 & 1.3e-14\\
               & 2,688  & 8.6e-3 & 2.0 & 2.2e-2 & 1.0 & 3.0e-14\\
               & 10,752 & 2.1e-3 & 2.0 & 1.1e-2 & 1.0 & 4.4e-14\\
               & 43,088 & 5.4e-4 & 2.0 & 5.4e-3 & 1.0 & 4.7e-13\\
        \multicolumn{7}{c}{} \\
        $P^2$--$P^1$
               & 672    & 9.4e-4 & 3.1 & 2.7e-3 & 1.8 & 5.2e-14\\
               & 2,688  & 1.2e-4 & 3.0 & 7.5e-4 & 1.9 & 9.3e-14\\
               & 10,752 & 1.4e-5 & 3.0 & 2.0e-4 & 1.9 & 1.9e-13\\
               & 43,088 & 1.7e-6 & 3.0 & 5.0e-5 & 2.0 & 4.3e-13\\
        \multicolumn{7}{c}{} \\
        $P^5$--$P^4$
               & 672    & 4.3e-8  & 6.0 & 1.69e-7  & 5.0 & 4.4e-13\\
               & 2,688  & 6.7e-10 & 6.0 & 5.23e-9  & 5.0 & 9.1e-13\\
               & 10,752 & 1.1e-11 & 6.0 & 1.62e-10 & 5.0 & 1.8e-12\\
        \hline
      \end{tabular}
    \label{tab:kovasznay}
  \end{center}
}
\end{table}

\subsection{Pressure robustness}
\label{ss:testing_pressure_robustness}

Pressure robustness is demonstrated using the test case proposed
in~\cite[Section~5.1]{Lederer:2017}. On the domain $\Omega = (0, 1)
\times(0, 1)$ we consider boundary conditions and a source term such
that the exact solution is $u = \mathrm{curl} \zeta$, where $\zeta =
x_{1}^2 (x_{1} - 1)^2 x_{2}^2 (x_{2} - 1)^2$ and $p = x_{1}^5 + x_{2}^5
- 1/3$. We vary the viscosity $\nu$ and consider different orders of
polynomial approximation. It can be observed in
\cref{tab:pressure_robust} that the errors in the velocity for the HDG
and EDG--HDG methods are indeed independent of the pressure and
viscosity, as expected from \cref{thm:pressure_robust}. The lack of
pressure robustness for the EDG method is evident in
\cref{tab:pressure_robust} where it is clear (in bold) that the velocity
error increases as viscosity is decreased for the EDG method.

\begin{table}
  \caption{Computed velocity, pressure and velocity divergence errors in
    the $L^2$-norm for the pressure robustness test case for the HDG
    (H), EDG (E) and EDG--HDG (EH) methods for different orders of
    polynomial approximation. H-10 represents an HDG $P^1$--$P^0$
    discretization, etc.} {\small
    \begin{center}
      \begin{tabular}{cc|ccc|ccc}
        \multicolumn{2}{c|}{} & \multicolumn{3}{c|}{$\nu = 1$} & \multicolumn{3}{c}{$\nu = 10^{-6}$} \\
        \multicolumn{2}{c|}{} & \multicolumn{3}{c|}{} & \multicolumn{3}{c}{}\\
        \hline
        Cells & Method & $\norm{u-u_h}$ & $\norm{p-p_h}$ & $\norm{\nabla\cdot u_h}$
                              & $\norm{u-u_h}$ & $\norm{p-p_h}$ & $\norm{\nabla\cdot u_h}$ \\
        \hline
        131,072 & H-10  & 5.0e-7 & 2.4e-3 & 1.1e-15 & 5.0e-7 & 2.4e-3 & 7.6e-11 \\
        131,072 & E-10  & 9.2e-7 & 2.4e-3 & 1.0e-15 & {\bf 4.8e-3} & 2.4e-3 & 1.1e-10 \\
        131,072 & EH-10 & 1.1e-6 & 2.4e-3 & 1.1e-15 & 1.1e-6 & 2.4e-3 & 0.8e-11 \\
        \hline
        2,048 & H-21  & 4.7e-7 & 4.3e-4 & 6.0e-16 & 5.2e-7 & 3.2e-4 & 5.4e-11 \\
        2,048 & E-21  & 6.8e-7 & 5.2e-4 & 1.6e-15 & {\bf 4.6e-3} & 3.2e-4 & 6.4e-11 \\
        2,048 & EH-21 & 7.0e-7 & 5.3e-4 & 5.6e-16 & 7.4e-7 & 3.2e-4 & 6.6e-11 \\
        \hline
        128 & H-43  & 2.6e-7 & 3.9e-5 & 1.8e-15 & 2.6e-7 & 3.8e-6 & 1.0e-10 \\
        128 & E-43  & 2.7e-7 & 3.7e-5 & 1.7e-15 & {\bf 3.0e-5} & 3.8e-6 & 1.2e-10 \\
        128 & EH-43 & 2.7e-7 & 3.7e-5 & 1.7e-15 & 2.7e-7 & 3.8e-6 & 1.1e-10 \\
        \hline
      \end{tabular}
    \label{tab:pressure_robust}
  \end{center}
}
\end{table}

\subsection{Minimal regularity test}
\label{ss:testing_minimal_regularity}

We consider the Stokes problem on the L-shaped domain
$\Omega := \del{-1, 1}^2 \backslash \sbr{-1,0}\times\sbr{0,1}$ with
$\nu=1$ and $f = 0$, see, e.g.~\citep{Hansbo:2008,Verfurth:1989}. The
Dirichlet boundary data are interpolated from the exact solution,
which in polar coordinates is given by:
\begin{subequations}
  \label{eq:cornerflow}
  \begin{align}
    u_x &= r^{\lambda}\sbr{(1+\alpha)\sin(\varphi)\psi(\varphi) + \cos(\varphi)\partial_{\varphi}\psi(\varphi)}
    \\
    u_y &= r^{\lambda}\sbr{-(1+\alpha)\cos(\varphi)\psi(\varphi) + \sin(\varphi)\partial_{\varphi}\psi(\varphi)}
    \\
    p &= -r^{\lambda-1}\sbr{(1+\lambda)^2 \partial_{\varphi}\psi(\varphi)
      + \partial^3_{\varphi}\psi(\varphi)}/(1-\lambda),
  \end{align}
\end{subequations}
where
\begin{multline}
  \label{eq:defpsi}
  \psi(\varphi) = \sin((1+\lambda)\varphi)\cos(\lambda\omega)/(1+\lambda) - \cos((1+\lambda)\varphi) \\
  -\sin((1-\lambda)\varphi)\cos(\lambda\omega)/(1-\lambda) + \cos((1-\lambda)\varphi),
\end{multline}
and where $\omega = 3\pi/2$ and $\lambda \approx 0.54448373678246$. Note
that $u \notin \sbr{H^2(\Omega)}^2$ and $p \notin H^1(\Omega)$ for this
problem.

\Cref{fig:Lshapeddomain} presents the computed velocity and pressure
errors for $P^1$--$P^0$ and $P^4$--$P^3$ discretizations against the
total number of degrees-of-freedom. The solutions are observed to
converge, and the velocity and pressure errors are approximately of
order $\mathcal{O}(h)$ and $\mathcal{O}(h^{1/2})$, respectively. This is
an example where use of the very simple $P^{1}$--$P^{0}$ discretization
could be appealing.

\begin{figure}
  \centering
  \subfloat[Velocity error.]{\includegraphics[width=0.7\textwidth]{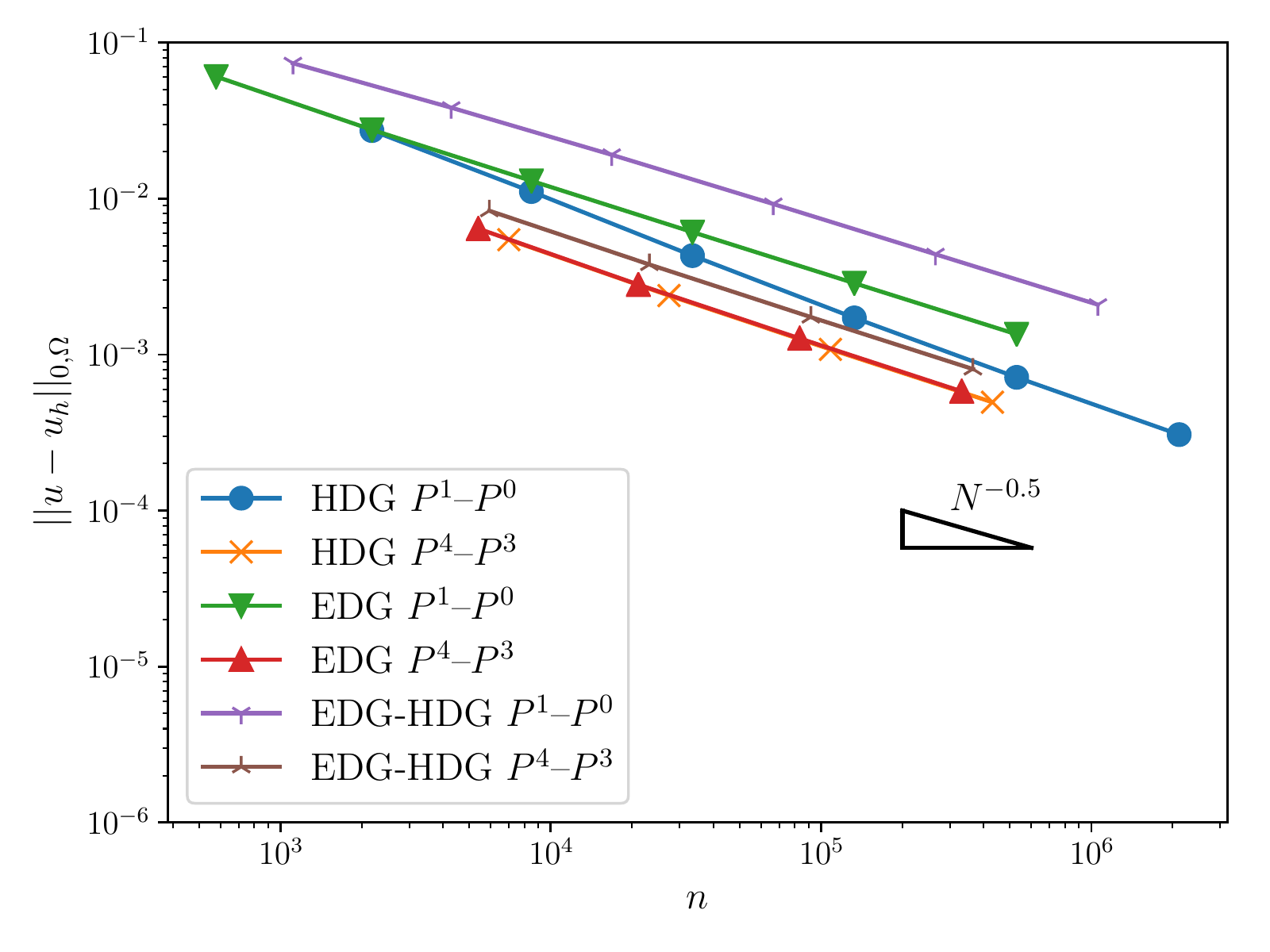}} \\
  \subfloat[Pressure error.]{\includegraphics[width=0.7\textwidth]{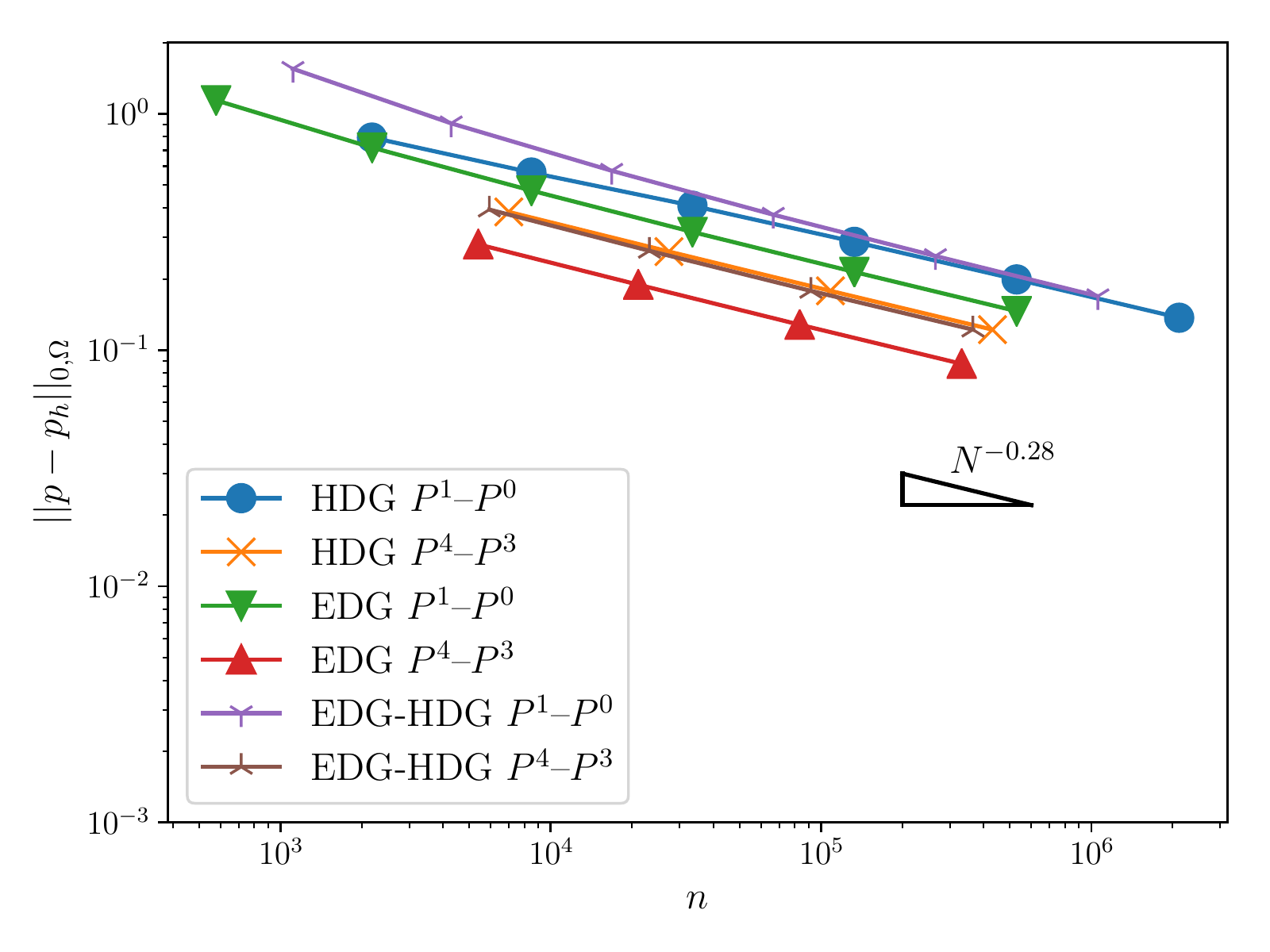}}
  \caption{Computed velocity and pressure errors in the $L^2$-norm
    versus total number of degrees-of-freedom~($n$) for the minimal
    regularity solution test case.}
  \label{fig:Lshapeddomain}
\end{figure}

\subsection{Preconditioned linear solvers}
\label{sec:preconditioning}

A motivation for considering EDG--HDG and EDG methods is efficiency when
combined with preconditioned iterative solvers, and in particular the
similarity of EDG--HDG and EDG to continuous methods for which a range
of solvers are known to perform well. In~\citep{Rhebergen:2018b} we
introduced an optimal preconditioner for the statically condensed
(cell-wise velocity eliminated locally) linear system obtained from the
HDG discretization of the Stokes problem, and the analysis holds also
for the EDG and EDG--HDG methods. The preconditioner is presented here
and we refer to \citep{Rhebergen:2018b} for the analysis.

The discrete problem for \cref{eq:discrete_problem} has the form:
\begin{equation}
  \label{eq:reorg-discsys}
  \begin{bmatrix}
    A_{uu} & \mathsf{B}^T
    \\
    \mathsf{B} & \mathsf{C}
  \end{bmatrix}
  \begin{bmatrix}
    u
    \\
    \mathsf{U}
  \end{bmatrix}
  =
  \begin{bmatrix}
    L_u
    \\
    \mathsf{L}
  \end{bmatrix},
  \qquad
  \mathsf{U} :=
  \begin{bmatrix}
    \bar{u}
    \\
    p
    \\
    \bar{p}
  \end{bmatrix} \quad
  \mathsf{L} :=
  \begin{bmatrix}
    L_{\bar{u}}
    \\
    0
    \\
    0
  \end{bmatrix},
\end{equation}
with
\begin{equation}
  \label{eq:def_sfA_sfB}
  \mathsf{B} :=
  \begin{bmatrix}
    A_{\bar{u}u}
    \\
    B_{pu}
    \\
    B_{\bar{p}u}
  \end{bmatrix},
  \quad
  \mathsf{C} :=
  \begin{bmatrix}
    A_{\bar{u}\bar{u}} & 0 & 0 \\ 0 & 0 & 0 \\ 0 & 0 & 0
  \end{bmatrix}.
\end{equation}
Here $u \in \mathbb{R}^{n_u}$ and $\bar{u}\in\mathbb{R}^{\bar{n}_u}$ are
the vectors of the discrete velocity with respect to the basis for the
cell-wise and facet velocities, respectively, and $p \in
\mathbb{R}^{n_p}$ and $\bar{p} \in \mathbb{R}^{\bar{n}_p}$ are the
vectors of the discrete pressure with respect to the basis for the
cell-wise and facet pressures, respectively. Furthermore, $A_{uu}$,
$A_{\bar{u}u}$ and $A_{\bar{u}\bar{u}}$ are the matrices obtained from
the discretization of $a_h((\cdot,0),(\cdot,0))$, $a_h((\cdot,0),
(0,\cdot))$ and $a_h((0,\cdot), (0,\cdot))$, respectively, and $B_{pu}$
and $B_{\bar{p}u}$ are the matrices obtained from the discretization of
$b_h((\cdot,0),(\cdot,0))$ and $b_h((0,\cdot),(\cdot,0))$. Noting that
$A_{uu}$ is a block diagonal matrix (one block per cell), it is possible
to efficiently eliminate $u$ from \cref{eq:reorg-discsys} using $u =
A_{uu}^{-1}\del{L_u - \mathsf{B}^T\mathsf{U}}$. This results in a
reduced system for $\mathsf{U}$ only,
\begin{equation}
  \label{eq:system_S}
  \begin{bmatrix}
    -A_{\bar{u}u}A_{uu}^{-1}A_{\bar{u}u}^T + A_{\bar{u}\bar{u}} & -A_{\bar{u}u}A_{uu}^{-1}B_{pu}^T
        & -A_{\bar{u}u}A_{uu}^{-1}B_{\bar{p}u}
    \\
    -B_{pu}A_{uu}^{-1}A_{\bar{u}u}^T & -B_{pu}A_{uu}^{-1}B_{pu}^T
        & -B_{pu}A_{uu}^{-1}B_{\bar{p}u}^T
    \\
    -B_{\bar{p}u}A_{uu}^{-1}A_{\bar{u}u}^T & -B_{\bar{p}u}A_{uu}^{-1}B_{pu}^T
        & -B_{\bar{p}u}A_{uu}^{-1}B_{\bar{p}u}^T
  \end{bmatrix}
  \begin{bmatrix}
    \bar{u}
    \\
    p
    \\
    \bar{p}
  \end{bmatrix}
  =
  \begin{bmatrix}
    L_{\bar{u}} - A_{\bar{u}u}A_{uu}^{-1}L_u
    \\
    -B_{pu}A_{uu}^{-1}L_u
    \\
    -B_{\bar{p}u}A_{uu}^{-1}L_u
  \end{bmatrix}.
\end{equation}

In \citep{Rhebergen:2018b} we introduced three optimal preconditioners
for the reduced form of the hybrid discretizations of the Stokes
problem: two block diagonal preconditioners and a block symmetric
Gauss--Seidel preconditioner. We discuss here only the block symmetric
Gauss--Seidel preconditioner. Let $\mathcal{P}_D$ and $\mathcal{P}_L$
be, respectively, the block-diagonal and the strictly lower triangular
part of the system matrix in \cref{eq:system_S}. The block symmetric
Gauss--Seidel preconditioner is then given by
\begin{equation}
  \label{eq:preconditioner}
  \mathcal{P} = (\mathcal{P}_L + \mathcal{P}_D)\mathcal{P}_D^{-1}(\mathcal{P}_L^T + \mathcal{P}_D).
\end{equation}
As discussed in \citep{Rhebergen:2018b}, algebraic multigrid can
successfully be applied to approximate the inverse of~$\mathcal{P}_D$.

Optimality of the preconditioner in \cref{eq:preconditioner} for the
HDG, EDG--HDG and EDG methods is tested for $P^1$--$P^0$ and
$P^2$--$P^1$ polynomial approximations. In all cases we use MINRES for
the outer iterations, with AMG (four multigrid V-cycles) to
approximate the inverse of each block of~$\mathcal{P}_D$. In all
cases, the solver is terminated once the relative true residual
reaches a tolerance of~$10^{-12}$.

We consider lid-driven cavity flow in the unit square
$\Omega = [-1, 1]^2$ and a cube $\Omega = [0, 1]^3$, using
unstructured simplicial meshes. Dirichlet boundary conditions are
imposed on $\partial\Omega$.  In two dimensions, $u = (1-x_1^4, 0)$ on
the boundary $x_2 = 1$ and the zero velocity vector on remaining
boundaries. In three dimensions we impose
$u = (1 - \tau_1^4, \tfrac{1}{10}(1-\tau_2^4), 0)$, with
$\tau_i = 2 x_i-1$, on the boundary $x_3 = 1$ and the zero velocity
vector on remaining boundaries. We set $\nu =
1$. \Cref{tab:preconditioning_2d} presents the number of iterations
for the two-dimensional problem and \cref{tab:preconditioning_3d}
presents the number of iterations for the three-dimensional
problem. The preconditioner in \cref{eq:preconditioner} is observed to
be optimal for all methods in both two and three dimensions -- the
iteration count is independent of the problem size, or at worst
exhibits a weak growth with increasing problem size. In all cases the
solver converges in fewer iterations for the EDG and EDG--HDG methods
compared to the HDG method. For example, on the finest grid in three
dimensions, using a $P^2$--$P^1$ discretization, HDG requires 300
iterations to converge, compared to 132 for EDG and 151 for EDG--HDG.

\begin{table}
  \caption{Iteration counts for preconditioned MINRES for the relative true residual
    to reach a tolerance of $10^{-12}$ for the lid-driven cavity problem
    in two dimensions.} {\small
    \begin{center}
      \begin{tabular}{c|cc|cc|cc}
        \multicolumn{1}{c}{} & \multicolumn{6}{c}{$P^1$--$P^0$} \\
        \hline
        & \multicolumn{2}{c|}{HDG} & \multicolumn{2}{c|}{EDG} & \multicolumn{2}{c}{EDG--HDG} \\
        Cells & DOFs & Its & DOFs & Its & DOFs & Its \\
        \hline
        176    & 1,892   & 204 & 509     & 177 & 970     & 193\\
        704    & 7,304   & 217 & 1,895   & 194 & 3,698   & 211\\
        2,816  & 28,688  & 230 & 7,307   & 200 & 14,434  & 212\\
        11,264 & 113,696 & 234 & 28,691  & 191 & 57,026  & 198\\
        45,056 & 452,672 & 236 & 113,699 & 193 & 226,690 & 197\\
        \hline
        \multicolumn{1}{c}{} & \multicolumn{6}{c}{$P^2$--$P^1$} \\
        \hline
        & \multicolumn{2}{c|}{HDG} & \multicolumn{2}{c|}{EDG} & \multicolumn{2}{c}{EDG--HDG} \\
        Cells & DOFs & Its & DOFs & Its & DOFs & Its \\
        \hline
        176    & 3,102   & 156 & 1,719   & 129 & 2,180   & 129\\
        704    & 12,012  & 166 & 6,603   & 131 & 8,406   & 131\\
        2,816  & 47,256  & 170 & 25,875  & 132 & 33,002  & 131\\
        11,264 & 187,440 & 182 & 102,435 & 130 & 130,770 & 129\\
        45,056 & 746,592 & 184 & 407,619 & 127 & 520,610 & 126\\
        \hline
      \end{tabular}
    \label{tab:preconditioning_2d}
  \end{center}
}
\end{table}

\begin{table}
  \caption{Iteration counts for preconditioned MINRES for the relative true residual
    to reach a tolerance of $10^{-12}$ for the lid-driven cavity problem
    in three dimensions.} {\small
    \begin{center}
      \begin{tabular}{c|cc|cc|cc}
        \multicolumn{1}{c}{} & \multicolumn{6}{c}{$P^1$--$P^0$} \\
        \hline
        & \multicolumn{2}{c|}{HDG} & \multicolumn{2}{c|}{EDG} & \multicolumn{2}{c}{EDG--HDG} \\
        Cells & DOFs & Its & DOFs & Its & DOFs & Its \\
        \hline
        524     & 14,540  & 258 & 1,180  & 169 & 4,520   & 190\\
        4,192   & 110,560 & 302 & 8,076  & 189 & 33,697  & 218\\
        33,536  & 861,440 & 327 & 59,988 & 189 & 260,351 & 210\\
        \hline
        \multicolumn{1}{c}{} & \multicolumn{6}{c}{$P^2$--$P^1$} \\
        \hline
        & \multicolumn{2}{c|}{HDG} & \multicolumn{2}{c|}{EDG} & \multicolumn{2}{c}{EDG--HDG} \\
        Cells & DOFs & Its & DOFs & Its & DOFs & Its \\
        \hline
        524    & 30,128    & 227 & 5,980   & 136 & 12,017  & 151 \\
        4,192  & 229,504   & 265 & 43,220  & 136 & 89,791  & 161 \\
        33,536 & 1,789,952 & 300 & 328,868 & 132 & 694,139 & 151 \\
        \hline
      \end{tabular}
    \label{tab:preconditioning_3d}
  \end{center}
}
\end{table}

\subsection{Performance comparison}
\label{sec:perform_comp}

We compare the overall performance of the HDG, EDG--HDG and EDG methods
in terms of solution time for a given level of accuracy using a problem
with $\nu = 1$ on the unit cube $\Omega = \sbr{0,1}^3$ with source and
Dirichlet boundary conditions such that the exact solution is given by
\begin{equation}
  u=\pi
  \begin{bmatrix}
    \sin(\pi x_1)\cos(\pi x_2) - \sin(\pi x_1)\cos(\pi x_3) \\
    \sin(\pi x_2)\cos(\pi x_3) - \sin(\pi x_2)\cos(\pi x_1) \\
    \sin(\pi x_3)\cos(\pi x_1) - \sin(\pi x_3)\cos(\pi x_2)
  \end{bmatrix},
  \quad
  p = \sin(\pi x)\sin(\pi y)\sin(\pi z) - 8/\pi^3.
\end{equation}
Meshes are composed of unstructured tetrahedral cells and generated
using Gmsh. We apply GMRES with restarts after 30 iterations, with the
preconditioner in \cref{sec:preconditioning} applied. The iterative
method is terminated once the relative true residual reaches~$10^{-12}$.

The performance results for $P^2$--$P^1$ and $P^3$--$P^2$
discretizations are presented in \cref{tab:relativeresults}. We
observe that the velocity error is approximately 1.2--1.6 times higher
for the EDG--HDG and EDG methods when compared to the HDG method on
the same mesh. However, the time to compute the solution using the
EDG--HDG or EDG method is substantially lower compared to the HDG
discretization.  This is due to the global linear systems for the
EDG--HDG and EDG methods being significantly smaller for a given mesh,
and the systems solving in fewer iterations. Particularly noteworthy
is the $P^2$--$P^1$ EDG simulation on the finest mesh, for which
compared to the $P^2$--$P^1$ HDG solution the error is $1.4$ times
greater in the $L_{2}$-norm but the solution time is just~$1/20$th.

\begin{table}
  \caption{Results are normalized with respect to the results of HDG on
    each mesh (in brackets). Here $n$ is the total number of degrees of
    freedom of the global system after static condensation.}
  \centering
  \begin{tabular}{cl|llll}
    \multicolumn{2}{c}{} & \multicolumn{4}{c}{$P^2$--$P^1$} \\
    \hline
    Mesh & Method & Rel. $\norm{u-u_h}$ & Rel. $n$ & Rel. its & Rel. time \\
    \hline
    1            & HDG      & 1 (2.0e-2)  & 1 (\num{30128})  & 1 (51)   & 1 (7.5 \si{\second}) \\
    (524 cells)  & EDG--HDG & 1.59        & 0.4              & 0.88     & 0.27 \\
                         & EDG      & 1.49        & 0.2              & 0.80     & 0.12 \\
    \hline
    2                   & HDG      & 1 (4.8e-3)  & 1 (\num{229504})  & 1 (55)   & 1 (79 \si{\second}) \\
    (\num{4192} cells)  & EDG--HDG & 1.58        & 0.4               & 0.76     & 0.12 \\
                         & EDG      & 1.45        & 0.2               & 0.65     & 0.08 \\
    \hline
    3                    & HDG      & 1 (1.0e-3)  & 1 (\num{1789953})  & 1 (58)   & 1 (770 \si{\second}) \\
    (\num{33536} cells)  & EDG--HDG & 1.51        & 0.4                & 0.55     & 0.09 \\
                         & EDG      & 1.38        & 0.2                & 0.52     & 0.06 \\
    \hline
    \multicolumn{2}{c}{} & \multicolumn{4}{c}{$P^3$--$P^2$} \\
    \hline
    Mesh & Method & Rel. $\norm{u-u_h}$ & Rel. $n$ & Rel. its & Rel. time \\
    \hline
    1            & HDG      & 1 (1.4e-3)  & 1 (\num{51960})  & 1 (64)   & 1 (22 \si{\second}) \\
    (524 cells)  & EDG--HDG & 1.38        & 0.5              & 0.72     & 0.17  \\
                         & EDG      & 1.30        & 0.3              & 0.67     & 0.13  \\
    \hline
    2                   & HDG      & 1 (2.6e-4)  & 1 (\num{396480})  & 1 (68)   & 1 ( \si{\second}) \\
    (\num{4192} cells)  & EDG--HDG & 1.52        & 0.5               & 0.59     & 0.13 \\
                         & EDG      & 1.20        & 0.3               & 0.57     & 0.10 \\
    \hline
    3                    & HDG      & 1 (3.2e-5) & 1 (\num{3095040})  & 1 (69)   & 1 (2105 \si{\second}) \\
    (\num{33536} cells)  & EDG--HDG & 1.23       & 0.5                & 0.51     & 0.14 \\
                         & EDG      & 1.15       & 0.3                & 0.49     & 0.11 \\
    \hline
  \end{tabular}
  \label{tab:relativeresults}
\end{table}

\section{Conclusions}
\label{sec:conclusions}

We have introduced and analyzed a new embedded--hybridized discontinuous
Galerkin (EDG--HDG) finite element method for the Stokes problem. The
analysis is unified in that it also covers the previously presented
hybridized (HDG) and embedded discontinuous Galerkin (EDG) methods for
the Stokes problem. All three methods are stable, have optimal rates of
convergence and satisfy the continuity equation pointwise. Only the HDG
and the EDG--HDG methods have velocity fields that are $H({\rm
div})$-conforming, and it is proved that a consequence of this is that
velocity error estimates are independent of the pressure. The analysis
results are supported by numerical experiments. Noteworthy is that the
analysis holds for the extremely simple piecewise linear/constant pair
for velocity/pressure field. The work was motivated by the question of
whether the EDG--HDG method could preserve the attractive features of
the HDG formulation and be more amenable to fast iterative solvers. This
has been shown to be the case, supported by analysis and numerical
examples.  Numerical examples demonstrate optimality of a carefully
constructed preconditioner, and for a given accuracy the EDG--HDG method
is considerably faster than the HDG method.
\subsubsection*{Acknowledgements}

SR gratefully acknowledges support from the Natural Sciences and
Engineering Research Council of Canada through the Discovery Grant
program (RGPIN-05606-2015) and the Discovery Accelerator Supplement
(RGPAS-478018-2015).

\appendix
\section{Interpolation estimate}
\label{sec:interpolation_estimate}

\begin{lemma}
  \label{lem:interpolationbound}
  For $v \in \sbr[1]{H^1(\Omega)}^{d}$, let $\Pi \boldsymbol{v} =
  (\Pi_{\rm BDM} v, \Pi_{L^{2}(\Gamma_{0})} v)$ where $\Pi_{\rm BDM}$ is
  the BDM interpolation operator in \cref{lem:BDMprojection}, and
  $\Pi_{L^{2}(\Gamma_{0})}$ is the $L^2$-projection into $\bar{V}_h$.
  Then
  \begin{equation}
    \label{eq:interpolationbound}
    \tnorm{\boldsymbol{v} - \Pi \boldsymbol{v}}_{v'} \le ch^k\norm{v}_{k+1,\Omega}.
  \end{equation}
\end{lemma}
\begin{proof}
  By definition,
  \begin{multline}
    \label{eq:boundingvprime2}
    \tnorm{\boldsymbol{v} - \Pi \boldsymbol{v}}_{v'}^2 = \sum_{K \in \mathcal{T}}
    \norm{\nabla(v - \Pi_{\rm BDM} v)}_K^2 + \sum_{K \in \mathcal{T}}
    \frac{\alpha_v}{h_K} \norm{\Pi_{\rm BDM} v - \Pi_{L^{2}(\Gamma_{0})} v}_{\partial K}^2
    \\
    + \sum_{K \in \mathcal{T}}
    \frac{h_K}{\alpha_v}\norm{\nabla(v - \Pi_{\rm BDM} v)\cdot
      n}_{\partial K}^2.
  \end{multline}
  We will bound each term on the right-hand side of
  \cref{eq:boundingvprime2} separately.

  By \cref{lem:BDMprojection} \cref{item:bdm_inverse_inequality},
  \begin{equation}
    \sum_{K\in\mathcal{T}} \norm{\nabla(v - \Pi_{\rm BDM}v)}_K^2 \le ch^{2k}\norm{v}^2_{k+1, \Omega}.
  \end{equation}

  By the triangle inequality
  \begin{multline}
    \label{eq:triangleT2}
    \sum_{K\in\mathcal{T}}\frac{\alpha_v}{h_K}\norm[0]{\Pi_{\rm BDM}v - \Pi_{L^{2}(\Gamma_{0})} v}_{\partial K}^2
    \\
    \le
    \sum_{K\in\mathcal{T}}\frac{\alpha_v}{h_K}\norm{\Pi_{\rm BDM}v - v}_{\partial K}^2
    +
    \sum_{K\in\mathcal{T}}\frac{\alpha_v}{h_K}\norm[0]{v - \Pi_{L^{2}(\Gamma_{0})} v}_{\partial K}^2.
  \end{multline}
  Applying a continuous trace inequality to the first term on the right
  hand side of \cref{eq:triangleT2}, and by \cref{lem:BDMprojection}
  \cref{item:bdm_inverse_inequality},
  \begin{equation}
    \begin{split}
      \sum_{K\in\mathcal{T}}\frac{\alpha_v}{h_K}\norm{\Pi_{\rm BDM}v - v}_{\partial K}^2
      &\le
      c \sum_{K\in\mathcal{T}}\del{h_K^{-2}\norm{\Pi_{\rm BDM}v - v}_{K}^2 + \envert{\Pi_{\rm BDM}v - v}_{1,K}^2}
      \\
      &\le
      ch^{2k}\norm{v}_{k+1, \Omega}^2.
    \end{split}
  \end{equation}
  Similarly, applying a continuous trace inequality to the second term
  on the right-hand side of \cref{eq:triangleT2}, and properties of the
  $L^2$-projection operator (e.g.~\citep{Pietro:book}),
  \begin{equation}
    \label{eq:final_term2}
    \begin{split}
      \sum_{K\in\mathcal{T}}\frac{\alpha_v}{h_K}\norm[0]{\Pi_{L^{2}(\Gamma_{0})} v - v}_{\partial K}^2
      &\le
      c \sum_{K\in\mathcal{T}}\del{h_K^{-2}\norm[0]{\Pi_{L^{2}(\Gamma_{0})} v - v}_{K}^2 + \envert[0]{\Pi_{L^{2}(\Gamma_{0})} v - v}_{1,K}^2}
      \\
      &\le
      ch^{2k}\norm{v}_{k+1,\Omega}^2.
    \end{split}
  \end{equation}
  %

  Finally, by a continuous trace inequality and \cref{lem:BDMprojection}
  \cref{item:bdm_inverse_inequality},
  \begin{equation}
    \begin{split}
    \sum_{K\in\mathcal{T}} \frac{h_K}{\alpha_v}\norm{\nabla(v - \Pi_{\rm BDM}v) \cdot n}_{\partial K}^2
    &\le \sum_{K\in\mathcal{T}}c\del{\envert{v-\Pi_{\rm BDM}v}_{1,K}^2 + h_K^2 \envert{v-\Pi_{\rm BDM}v}_{2,K}^2}
    \\
    &\le
      ch^{2k}\norm{v}^2_{k+1, \Omega}.
    \end{split}
  \end{equation}

  The result follows by combining the bounds for each term in
  \cref{eq:boundingvprime2}. \qed
\end{proof}

\section{Pressure-robust $L^2$ error estimate}
\label{eq:l2normerrorestimate}

Following the steps to prove \cref{thm:pressure_robust} and exploiting
equivalence of the $\tnorm{\cdot}_{v}$ and $\tnorm{\cdot}_{v'}$ norms on
$V_{h} \times \bar{V}_{h}$ leads to the following:
\begin{corollary}[Approximation in the $\tnorm{\cdot}_{v'}$ norm]
  \label{cor:error_prime_norm}
  Let $u \in \sbr{H^{k + 1}(\Omega)}^d$ be the velocity solution of the
  Stokes problem \cref{eq:stokes} with $k \ge 1$, let $\boldsymbol{u} =
  (u, u)$, and let $\boldsymbol{u}_h \in X_h^{v}$ be the velocity
  solution of the finite element problem \cref{eq:discrete_problem} for
  the HDG or EDG--HDG formulations. Then
  \begin{equation}
    \tnorm{\boldsymbol{u} - \boldsymbol{u}_h}_{v'} \le c h^k \norm{u}_{k+1, \Omega}.
  \end{equation}
\end{corollary}

To prove a velocity error estimate in the $L^2$-norm we will rely on the
following regularity assumption. If $(u,p)$ solves the Stokes problem
\cref{eq:stokes} for $f \in \sbr[1]{L^2(\Omega)}^d$, we have on a convex
polygonal domain
\begin{equation}
  \label{eq:regularityassumption}
  \nu \norm{u}_{2,\Omega} + \norm{p}_{1,\Omega} \le c_r\norm{f}_{\Omega},
\end{equation}
where $c_r$ is a constant~\cite[Chapter~II]{Girault:book}.

\begin{lemma}[Boundedness of $a_{h}$ on the extended space] There exists
  a $C_a > 0$, independent of $h$, such that for all $\boldsymbol{u} \in
  V(h) \times \bar{V}(h)$ and for all $\boldsymbol{u} \in V(h) \times
  \bar{V}(h)$
  \begin{equation}
    \label{eq:bound_ah_infinite}
    \envert{a_h(\boldsymbol{u}, \boldsymbol{v})}
    \le C_a \nu \tnorm{\boldsymbol{u}}_{v'} \tnorm{\boldsymbol{v}}_{v'}.
  \end{equation}
\end{lemma}
The proof of this is identical to that
for~\cite[Lemma~4.3]{Rhebergen:2017}.

\begin{theorem}[Pressure robust velocity error estimate in the
  $L^2$-norm] Let $(u,p) \in \sbr[1]{H^{k + 1}(\Omega)}^d \times
  H^{k}(\Omega)$ solve the Stokes problem \cref{eq:stokes} with $k \ge
  1$, and let $\boldsymbol{u} = (u, u)$ and $\boldsymbol{p} = (p, p)$.
  If $(\boldsymbol{u}_h, \boldsymbol{p}_h) \in X_h$ solves the finite
  element problem \cref{eq:discrete_problem} for the HDG or EDG--HDG
  formulation then, subject to the regularity condition in
  \cref{eq:regularityassumption}, there exists a constant $C_V>0$,
  independent of $h$, such that
  \begin{equation}
    \label{eq:pressure_robust_l2}
    \norm{u - u_{h}}_{\Omega} \le C_{V} h^{k+1} \norm{u}_{k + 1, \Omega}.
  \end{equation}
\end{theorem}
\begin{proof}
  Let $(\zeta_{u}, \zeta_{p}) \in X$ solve the Stokes problem
  \cref{eq:stokes} for $f = (u - u_h)$. Then
  \begin{equation}
    a_{h}\del{(\zeta_{u}, \zeta_{u}), \boldsymbol{v} }
    +
    b_{h}\del{(\zeta_{p}, \zeta_{p}), v}
    =
    \int_{\Omega} (u - u_{h}) \cdot v \dif x
    \quad \forall \ \boldsymbol{v} \in  V(h) \times \bar{V}(h).
  \end{equation}
  Setting $\boldsymbol{v} = \boldsymbol{u} - \boldsymbol{u}_{h}$ and
  noting that $b_{h}\del{(\zeta_{p}, \zeta_{p}), v} = 0$ by the
  regularity of $\zeta_{p}$ and by $v$ being divergence-free and $H({\rm
  div})$-conforming, we have
  \begin{equation}
    a_{h}\del{(\zeta_{u}, \zeta_{u}), \boldsymbol{u} - \boldsymbol{u}_{h} }
    =
    \norm{u - u_{h}}^{2}_{\Omega}.
  \end{equation}
  Note also that
  \begin{equation}
    a_{h}\del{\boldsymbol{v}_{h}, \boldsymbol{u} - \boldsymbol{u}_{h} }
    = 0 \quad \forall \ \boldsymbol{v}_{h} \in X_{h}^{v}
  \end{equation}
  by adjoint consistency. Setting $\boldsymbol{v}_{h} = \Pi
  \boldsymbol{\zeta}_{u}$ where $\Pi$ is the projection in
  \cref{lem:interpolationbound}, then by boundedness of $a_{h}$
  \cref{eq:bound_ah_infinite},
  \begin{equation}
    \begin{split}
      \norm{u - u_{h}}_{\Omega}^{2}
      &=
      a_{h}\del{\boldsymbol{\zeta}_{u} - \Pi\boldsymbol{\zeta}_{u}, \boldsymbol{u} - \boldsymbol{u}_{h} }
      \\
      &\le
      C \nu
      \tnorm{\boldsymbol{\zeta}_{u} - \Pi\boldsymbol{\zeta}_{u}}_{v'}
      \tnorm{\boldsymbol{u} - \boldsymbol{u}_{h}}_{v'}
      \\
      &\le
      C \nu
      h \norm{\zeta_{u}}_{2}
      \tnorm{\boldsymbol{u} - \boldsymbol{u}_{h}}_{v'}
      \\
      &\le
      C
      h \norm{u - u_{h}}
      \tnorm{\boldsymbol{u} - \boldsymbol{u}_{h}}_{v'},
    \end{split}
  \end{equation}
  hence
  \begin{equation}
    \label{eq:finalequation}
    \norm{u - u_{h}}_{\Omega}
    \le
    C h \tnorm{\boldsymbol{u} - \boldsymbol{u}_{h}}_{v'}.
  \end{equation}
  The result follows from applying \cref{cor:error_prime_norm} to
  $\tnorm{\boldsymbol{u} - \boldsymbol{u}_{h}}_{v'}$. \qed
\end{proof}
\bibliographystyle{elsarticle-num-names}
\bibliography{references}
\end{document}